\documentclass{article}
\usepackage[utf8]{inputenc}
\usepackage[english]{babel}
\usepackage{blindtext}
\usepackage{amssymb}
\usepackage{amsthm}
\usepackage{amsmath}
\usepackage{tikz-cd}
\usepackage{setspace}
\usepackage{thmtools}
\usepackage{thm-restate}
\usepackage{hyperref}
\usepackage{cleveref}
\usepackage{booktabs}
\usepackage[margin=1.2in]{geometry}

\usepackage{graphicx}
\graphicspath{ {images/} }

\newcommand{\im}[1]{\mathrm{im} #1}

\theoremstyle{plain}
\newtheorem{theorem}{Theorem}[section]

\theoremstyle{definition}
\newtheorem{definition}[theorem]{Definition}

\theoremstyle{plain}
\newtheorem{prop}[theorem]{Proposition}

\theoremstyle{plain}
\newtheorem{corollary}[theorem]{Corollary}

\theoremstyle{plain}
\newtheorem{lemma}[theorem]{Lemma}

\theoremstyle{remark}
\newtheorem{remark}[theorem]{Remark}

\theoremstyle{remark}
\newtheorem{example}[theorem]{Example}

\theoremstyle{plain}
\newtheorem{conjecture}[theorem]{Conjecture}

\theoremstyle{plain}
\newtheorem{question}[theorem]{Question}

\title{Detected Seifert surfaces and intervals of left-orderable surgeries}
\author{Yi Wang}
\date{}

\begin{document}

\maketitle
	
\begin{abstract}
We give a geometric criterion for the existence of one-sided intervals of left-orderable Dehn fillings near the longitudinal slope. Let $M$ be an irreducible integral homology solid torus containing an essential free genus-one Seifert surface $S$ which is not a fiber. Suppose that the longitudinal filling $M(0)$ is irreducible and that the capped surface $\widehat S$ is its unique JSJ torus. We prove that if the complementary piece is Seifert-fibered, or hyperbolic with a real place in its trace field, then all sufficiently small rational fillings on one side of zero have left-orderable fundamental group. The side is determined by the sign of an explicit invariant of the limiting representation. The proof starts from a real holonomy representation of the complementary piece. Peripheral infinitesimal regularity produces a smooth character-matching branch on the cut-open handlebody, and one real half-branch glues to $SL_2(\mathbb R)$-representations of $\pi_1(M)$ whose characters approach an ideal point detecting $S$. Their peripheral evaluations form a nonhorizontal end of the zero-sheet of the holonomy extension locus asymptotic to the horizontal axis. As an application, we show that for $n \geq 1$, there exists $a_n > 0$ such that for all rational $r \in (0, a_n)$, the $r$-Dehn filling of the $(-3,3,2n+1)$-pretzel knot complement is left-orderable.
\end{abstract}

\section{Introduction}

We relate the geometry of longitudinal Dehn filling to the existence of nearby left-orderable fillings for manifolds with free genus-one Seifert surfaces. Starting from a real holonomy representation of the manifold obtained by cutting the longitudinal filling along the capped surface, we construct an arc of $SL_2(\mathbb R)$-characters approaching an ideal point that detects the original surface. The corresponding end of the holonomy extension locus yields a one-sided interval of left-orderable fillings.

\medskip

The left-orderability of 3-manifold groups is an important problem in low-dimensional topology; see Clay-Rolfsen \cite{clayrolfsen} for an exposition. It is closely connected to the $L$-space conjecture of Boyer-Gordon-Watson \cite{bgw}; extensive computational evidence is discussed in \cite{dunfieldcensus}. 

\medskip

Left-orderability of Dehn fillings has been established for several classes of knot exteriors. Zung \cite{zung} proved that every nontrivial surgery on the figure-eight knot has left-orderable fundamental group. Gao \cite{gaotwobridge}, Tran \cite{tran} and Le \cite{khanh} obtained orderability results for fillings of various two-bridge knots, while Nie \cite{niepretzel}, Khan-Tran \cite{khantran}, Le \cite{lethesis}, and Zhao \cite{zhao2026leftorderability} treated surgeries of families of three-strand pretzel knots. In particular, Zhao verifies the $L$-space conjecture for fillings of the $(-2, 3, 2q+1)$-pretzel knots with $q \geq 3$. The author \cite{conway} also utilized the techniques in the present paper to prove orderability results for sufficiently large positive and negative Dehn fillings for certain knot complements with essential Conway spheres. Related results for cyclic branched covers of knots were obtained by Hu \cite{hu}, Li-Nie \cite{linnie}, and Turner \cite{turner}. 

\medskip

A representation-theoretic approach to left-orderability of 3-manifolds was developed systematically by Culler-Dunfield \cite{cullerdunfield} and Gao \cite{gao}. Under suitable local hypotheses, real characters represented by $SL_2(\mathbb R)$ representations deform into real arcs. After resolving the lifting obstruction and controlling peripheral translation numbers, these arcs can yield intervals of left-orderable fillings. The starting characters include real Galois conjugates of discrete faithful representations into $SL_2(\mathbb{C})$, hyperbolic Dehn filling points, or points coming from roots of the Alexander polynomial.

\medskip

Theorem 1 of Herald-Zhang \cite{heraldzhang} proves that, for an irreducible knot exterior in an integral homology 3-sphere, a simple unit-circle root of the Alexander polynomial implies left-orderability of all sufficiently small nonzero rational fillings on both sides of zero. For a genus-one Seifert matrix $V$, put $d = \det V$. With the normalization $\Delta_K(1) = 1$, we have $\Delta_K(t) = d(t + t^{-1}) + 1 - 2d$. This polynomial has simple unit-circle roots precisely when $d > 0$. Thus, \cite{heraldzhang} already establishes two-sided orderability near zero in this case. In addition, the results of Gao \cite{gao} and Le \cite{lethesis} utilize positive real roots of the Alexander polynomial; in particular, \cite{lethesis} obtains orderability results for rational surgeries near zero for certain two-bridge knots and $(-3, 3, 2n+1)$ pretzel knots. The construction in this present paper provides a different source of real deformations, arising from ideal points detecting Seifert surfaces, and applies in particular to examples whose Alexander polynomials have no unit-circle roots. 

\medskip

Our construction uses limiting character techniques developed by Paoluzzi-Porti \cite{paoluzziporti}, Tillmann \cite{tillmann}, and the author \cite{wang}. The principal new step is to show that the geometry of the capped complement provides enough regularity at the limiting representation to construct a character-matching branch on the cut-open handlebody. We realize one side of this branch by $SL_2(\mathbb R)$-representations of the original manifold and identify its ideal endpoint. 

\medskip

Let $M$ be an integral homology solid torus, with a chosen meridian-longitude basis $(\mu, \ell)$, and write $M(r)$ for Dehn filling along the slope $r=p/q$ represented by $\mu^p\ell^q$. In particular, $M(0)$ denotes Dehn filling along the homological longitude. A Seifert surface $S \subset M$ is called \emph{free} if the compact manifold $H = M \setminus \setminus S$ obtained by cutting along $S$ is a handlebody. When $S$ has genus one, $H$ is a genus-two handlebody. Capping $S$ with a meridian disk of the filling solid torus gives a torus $\widehat S \subset M(0)$. The manifold $N = M(0) \setminus \setminus \widehat S$ has two torus boundary components and is obtained from $H$ by attaching a 2-handle along the boundary of $S$. We consider the case in which $\widehat S$ is the unique JSJ torus of $M(0)$. Our main theorem relates the geometry of $N$ to orderability of fillings of $M$ near slope zero. 

\begin{definition}
We say that $M$ is \emph{half-orderable near 0} if there exists some $a \in \mathbb{R} \setminus \{0\}$ such that $\pi_1(M(r))$ is left-orderable for every rational slope $r$ strictly between $0$ and $a$.
\end{definition}

The zero-sheet $H_{0,0}(M)$ of the holonomy extension locus records logarithmic absolute values of peripheral eigenvalues of lifted real representations whose meridional and longitudinal translation numbers vanish. Its definition is recalled in Section \ref{sec:background}.

\begin{theorem}\label{thm:mainthm}
Let $M$ be a compact, connected, orientable, irreducible integral homology solid torus, and let $S \subset M$ be an essential free genus-one Seifert surface which is not a fiber. Suppose that $M(0)$ is irreducible and that the capped surface $\widehat S$ is its unique JSJ torus. Set $N = M(0) \setminus \setminus \widehat S$. Assume that either $N$ is Seifert-fibered or the interior of $N$ is complete finite-volume hyperbolic and its trace field has a real embedding. Then $H_{0,0}(M)$ contains a nonhorizontal end asymptotic to the horizontal axis, and there is $a \neq 0$ such that $\pi_1(M(r))$ is left-orderable for every rational slope $r$ strictly between $0$ and $a$. 
\end{theorem}

The proof establishes a more general criterion for real representations of the capped complement, stated in Theorem \ref{thm:endpointcriterion}; the geometric
hypotheses above ensure that its peripheral compatibility and regularity conditions hold. Corollary \ref{cor:side} determines the side of the interval in terms of the peripheral data of the representation.

\medskip

The hypotheses of the theorem are expressed in terms of the topology of the Seifert surface and the geometry of the capped complement. Thus, the geometry of $N$ supplies real character arcs approaching ideal points of $X(M)$, and hence orderable fillings near the longitudinal slope. 

\medskip

There are two main difficulties in passing from the geometric representation of $N$ to a real deformation of $M$. First, the two surface restrictions of a representation of $H$ must have matching characters. We show that the regularity provided by the geometric limiting data forces a smooth branch in this matching locus, with nontrivial longitudinal variation. Second, matching characters does not by itself guarantee gluing in $SL_2(\mathbb R)$. Away from the endpoint, the irreducible surface restrictions are conjugate via a matrix which may have negative determinant. We show that the determinant of real conjugating matrix vanishes simply at the limiting character. It is therefore positive on one real half-branch, allowing determinant-one real gluing. We show that the characters approach an ideal point detecting $S$ and that every lift of the projectivized family has zero longitudinal translation number.

\medskip

For certain knot exteriors in $S^3$, Theorem \ref{thm:mainthm} verifies the predictions of the $L$-space conjecture for the resulting sufficiently small nonzero fillings. Indeed, Li-Roberts \cite{liroberts} construct coorientable taut foliations on all fillings in an open interval of slopes containing zero. By the obstruction of Ozsv\'ath-Szab\'o \cite{ozsvathszabo}, the nonzero rational fillings in this interval are not $L$-spaces. After shrinking the one-sided interval supplied by Theorem \ref{thm:mainthm}, its nonzero rational fillings therefore satisfy all three conditions predicted by the conjecture. The theorem applies, in particular, to the $(-3, 3, 2n+1)$ pretzel knots for $n \geq 1$, whose longitudinal fillings have the required single-torus decomposition, as shown by Sekino \cite{sekino}. Further examples are discussed in Section \ref{sec:examples}. 

\subsection{Outline of the paper}

Section \ref{sec:background} recalls character varieties, ideal points, and the holonomy extension locus, and establishes a criterion for obtaining one-sided intervals of left-orderable fillings from real ideal point ends. Section \ref{sec:deformation} fixes the topological setup, chooses compatible real holonomy representations of the capped complement, and verifies their peripheral regularity. In Section \ref{sec:proof}, we construct a smooth character-matching branch on the cut-open handlebody and glue one real half-branch to representations of $\pi_1(M)$. We identify the detecting ideal point approached by this family and prove that its longitudinal translation number vanishes, completing the proof of Theorem \ref{thm:mainthm}. In addition, we prove Corollary \ref{cor:side}, which determines the side of the interval. Section \ref{sec:examples} applies the theorem to odd pretzel knots, presents computed trace field data for hyperbolic capped complements, and discusses further questions. 

\subsection{Acknowledgements}

The author thanks Nathan Dunfield for numerous helpful conversations that guided him through this work. 

\subsection{AI Declaration}

AI tools were used to improve spelling, grammar, prose, and exposition. 

\section{Background}\label{sec:background}

For relevant background on 3-manifolds, JSJ decompositions, and Seifert fiberings, see \cite{thurston}. For more background on Culler-Shalen theory and limiting characters, see the background section of \cite{wang}, and for background on left-orderability and the holonomy extension locus, see \cite{gao}.

\subsection{Ideal points and incompressible surfaces}

Let $M$ be a compact connected orientable 3-manifold, and let $\Gamma = \pi_1(M)$. The \emph{representation variety} is 
\begin{equation*}
	R(M) = \mathrm{Hom}(\Gamma, SL_2(\mathbb C))
\end{equation*} 
For $\rho \in R(M)$, its \emph{character} is the function
\begin{equation*}
	\chi_\rho: \Gamma \to \mathbb C \ \ \ \ \ \chi_\rho(\gamma) = \mathrm{tr}(\rho(\gamma))
\end{equation*}

\begin{definition}
The \emph{$SL_2(\mathbb{C})$ character variety} of $M$ is the affine algebraic quotient
\begin{equation*}
	X(M) = R(M) // SL_2(\mathbb C)
\end{equation*}
Its complex points are identified with the characters $\chi_\rho$. For $\gamma \in \Gamma$, the associated regular trace function is 
\begin{equation*}
	I_\gamma: X(M) \to \mathbb C \ \ \ \ \ I_\gamma(\chi_\rho) = \chi_\rho(\gamma)
\end{equation*}
We use the reduced affine variety structure on $X(M)$; see \cite{cullershalen} for its full construction.
\end{definition}

Two representations have the same character precisely when their semisimplifications are conjugate. In particular, an irreducible character determines a unique conjugacy class, whereas a reducible character can have nonconjugate representations.

\begin{definition}
Suppose that the interior of $M$ carries a complete finite-volume hyperbolic structure, with geometric representation $\bar\rho_{hol}: \pi_1(M) \to PSL_2(\mathbb C)$, defined up to conjugacy. A \emph{holonomy representation} $\rho_0: \pi_1(M) \to SL_2(\mathbb{C})$ is a lift of $\bar\rho_{hol}$ to $SL_2(\mathbb C)$. If $M$ is Seifert-fibered over an orientable hyperbolic 2-orbifold $\mathcal O$, let $\bar\rho_{hol}$ be the composition $\pi_1(M) \to \pi_1^{orb}(\mathcal O) \to PSL_2(\mathbb R)$, where the second map is the holonomy of a chosen complete finite-area hyperbolic structure on the interior of $\mathcal O$. A \emph{holonomy representation} is a lift of this composition to $SL_2(\mathbb C)$, when such a lift exists. We say that a \emph{canonical component} is an algebraically irreducible component which contains the character of a Galois conjugate of a holonomy representation. 
\end{definition}

A foundational result of Thurston \cite{thurston} is that the complex dimension of the canonical component is equal to the number of cusps in a hyperbolic 3-manifold. One particular focus of this paper will be \emph{ideal points}, which can informally be thought of as ``points at infinity" on affine algebraic curves. 

\begin{definition}
Let $D$ be an irreducible affine algebraic curve, let $\overline D$ be its projective closure, and let $\nu_D: \widehat D \to \overline D$ be the normalization. An \emph{ideal point} of $D$ is a point of $\widehat D \setminus \nu_D^{-1}(D)$. 
\end{definition}

If $D \subset X(M)$, every trace function $I_\gamma$ is regarded as a meromorphic function on $\widehat D$. 

\medskip

Culler and Shalen \cite{cullershalen} established a method which associates ideal points on a curve in $X(M)$ to incompressible surfaces in 3-manifolds. For the full theory, see \cite{cullershalen} or Chapter 19 of \cite{handbook}, while brief outlines can be found in \cite{paoluzziporti} or \cite{wang}.

\begin{definition}
Let $D \subset X(M)$ be an irreducible affine curve and let $x_\infty \in \widehat D$ be an ideal point. An essential surface $S \subset M$ is \emph{detected} by $x_\infty$ if it is associated to this ideal point by the Culler-Shalen construction. 
\end{definition}

Determining which incompressible surfaces in a given 3-manifold $M$ are detected by ideal points in $X(M)$ is a subtle problem, which has been studied in \cite{Cooper1994PlaneCA}, \cite{paoluzziporti}, \cite{Chesebro2005NotAB}, for instance. Some developments in Culler-Shalen theory, initiated by Dunfield \cite{Dunfield2008AP1} and Tillmann \cite{tillmann} have centered on the \emph{limiting character} at $x$. See \cite{wang} for a detailed exposition of the limiting character and its relevance to Culler-Shalen theory.

\begin{definition}
Let $x_\infty$ be an ideal point of an irreducible curve $D \subset X(M)$, and let $\Gamma_0 \leq \pi_1(M)$ be a finitely generated subgroup. Suppose that $I_\gamma$ is regular at $x_\infty$ for every $\gamma \in \Gamma_0$. The \emph{limiting character} on $\Gamma_0$ is $\chi_{\infty,\Gamma_0}(\gamma) = I_\gamma(x_\infty)$. 
\end{definition}

This is a character because the restriction map is expressed by finitely many trace coordinates and the corresponding character variety is Zariski closed. Equivalently, restriction to $\Gamma_0$ extends across $x_\infty$ as a map to $X(\Gamma_0)$. An equivalent interpretation is that $\chi_{\infty, \Gamma_0}(\gamma) = \lim_{z \to x_\infty}I_\gamma(z)$. 

\medskip

If $S \subset M$ is a surface detected by $x_\infty$, Lemma 2.2.1 of \cite{cullershalen} implies that if $M_i$ is a path-component of $M \setminus S$, then for all $\gamma \in \pi_1(M_i)$, $I_\gamma$ is regular at $x_\infty$. There is thus a well-defined limiting character of $\pi_1(M_i)$. In \cite{wang} and \cite{twicepunctured}, the author initiated the study of limiting characters of punctured tori in hyperbolic knot complements. Limiting characters at ideal points detecting Conway spheres were also discussed in \cite{paoluzziporti}. We now state the relevant results on the limiting character that will be used for this paper. 

\medskip

It is important to distinguish a candidate limiting character from an ideal point of a global character curve. For a connected nonseparating surface $S$, write $H = M \setminus \setminus S$, and let $r: X(M) \to X(H)$ be the restriction map. Suppose $D \subset X(M)$ is an irreducible curve on which $r$ is nonconstant, and set $C = \overline{r(D)}$. Passing to smooth projective models gives a morphism $\widehat r: \widehat D \to \widehat C$. If $c_0 \in \widehat C$ lies over an affine character $\chi_0 \in C \setminus r(X(M))$, every point of $\widehat r^{-1}(c_0)$ is an ideal point of $D$: an affine preimage would contradict the choice of $\chi_0$. The affine fiber $r^{-1}(\chi_0)$ is empty, not itself an ideal point. 

\medskip

Identifying a surface detected by one of these ideal points requires a separate argument. For a connected essential surface we use the surface-to-tree criterion in Section 3.2 of \cite{tillmann} to identify the surface associated to the particular ideal point we construct. 

\medskip

The starting point of our construction is a character $\chi_0$ of an $SL_2(\mathbb R)$-representation of $\pi_1(H)$ that lies in the closure of the restriction image but is not realized by a representation of $\pi_1(M)$. Under the hypotheses established below, we construct a real arc in the restriction image and realize it by $SL_2(\mathbb R)$-representations of $\pi_1(M)$. We then show that the characters of this family approach a specified ideal point detecting $S$, and control the longitudinal translation number to obtain left-orderable Dehn fillings. 

\subsection{Left-orderability}

\begin{definition}
Let $G$ be a finitely generated group. We say $G$ is \emph{left-orderable} if it admits a strict total ordering $<$ such that for all $f, g, h \in G$, $g < h$ implies $fg < fh$. If $M$ is a 3-manifold, we say that $M$ is \emph{left-orderable} if $\pi_1(M)$ is a left-orderable group. By convention, we say the trivial group is not left-orderable.
\end{definition}

The interest in left-orderability in the context of low-dimensional topology comes from the \emph{$L$-space conjecture} from \cite{bgw}:

\begin{conjecture}
Let $M$ be a closed, connected, orientable, irreducible rational homology 3-sphere. Then the following are equivalent:
\begin{enumerate}
	\item $M$ is not an $L$-space.
	\item $M$ admits a coorientable taut foliation.
	\item $\pi_1(M)$ is left-orderable.
\end{enumerate}
\end{conjecture}

 Here, we recall some of the existing literature on left-orderability of closed 3-manifold groups, with a particular emphasis on the techniques involving the translation extension locus and the holonomy extension locus. In \cite{cullerdunfield}, a systematic way to study left-orderability of 3-manifolds using character varieties was developed. By Theorem 1.1 of \cite{wiest}, if $Y$ is a compact, connected, orientable, irreducible 3-manifold and there is a nontrivial homomorphism from $\pi_1(Y)$ to a left-orderable group, then $\pi_1(Y)$ is left-orderable. The overarching strategy revolves around $\widetilde{PSL_2(\mathbb{R})}$, the universal cover of $SL_2(\mathbb{R})$ and $PSL_2(\mathbb{R})$. 

\begin{lemma}\cite{ghys}
$\widetilde{PSL_2(\mathbb{R})}$ is left-orderable.
\end{lemma}

The strategy developed by \cite{cullerdunfield} utilizes $SL_2(\mathbb{C})$ character varieties of one-cusped hyperbolic 3-manifolds to find nontrivial homomorphisms from fundamental groups of Dehn fillings to $\widetilde{PSL_2(\mathbb{R})}$, showing that these Dehn fillings are also left-orderable. From $SL_2(\mathbb{C})$ character varieties, real points may correspond to $SL_2(\mathbb{R})$ representations. Given such a representation $\rho: \pi_1(M) \to SL_2(\mathbb{R})$, writing $\bar\rho$ for the quotient in $PSL_2(\mathbb R)$, the \emph{Euler class} $e(\bar\rho) \in H^2(\pi_1(M); \mathbb{Z})$ serves as an obstruction from lifting $\bar\rho$ to $\widetilde{PSL_2(\mathbb{R})}$. See \cite{ghys} for the definition, and \cite{cullerdunfield} for its role as an obstruction to lifting representations. For the purposes of this paper, we will restrict our attention to the case where $M$ is an integral homology solid torus, for which $H^2(\pi_1(M); \mathbb{Z})$ is trivial, and all $SL_2(\mathbb{R})$ representations will lift. 

\medskip

Identify $\mathbb RP^1$ with $\mathbb R/\mathbb Z$. We realize $\widetilde{PSL}_2(\mathbb R)$ as the group of lifts to $\mathbb R$ of the projective action of $PSL_2(\mathbb R)$; its center consists of the deck translations $x \mapsto x + n$, $n \in \mathbb Z$. 

\begin{definition}
The \emph{translation number} of a group element $g \in \widetilde{PSL_2(\mathbb{R})}$ is defined as 
\begin{equation*}
	\text{trans}(g) = \lim_{n \to \infty}\frac{g^n(x) - x}{n} \ \ \ \ \ x \in \mathbb{R}
\end{equation*}
This is independent of the choice of real number $x$.
\end{definition}

The translation number is continuous on $\widetilde{PSL}_2(\mathbb R)$ and is integral on hyperbolic, parabolic, and central elements; see Section 3 of \cite{cullerdunfield}.

\begin{definition}
A matrix $A \in SL_2(\mathbb{R})$ is:
\begin{enumerate}
	\item \emph{elliptic} if $|\text{tr}(A)| < 2$.
	\item \emph{parabolic} if $|\text{tr}(A)| = 2$ and $A \neq \pm I$.
	\item \emph{hyperbolic} if $|\text{tr}(A)| > 2$.
	\item \emph{central} if $A = \pm I$.
\end{enumerate}
An element $g \in \widetilde{PSL_2(\mathbb{R})}$ is \emph{elliptic, parabolic, hyperbolic, or central} if it descends to a matrix with the corresponding adjective in $SL_2(\mathbb{R})$.
\end{definition}

In order to capture the relevant information of $\widetilde{PSL_2(\mathbb{R})}$ representations, \cite{cullerdunfield} defined the \emph{translation extension locus} which captures $\widetilde{PSL_2(\mathbb{R})}$ representations which are elliptic or parabolic on the boundary, and \cite{gao} defined the \emph{holonomy extension locus}, which captures $\widetilde{PSL_2(\mathbb{R})}$ representations which are hyperbolic or parabolic on the boundary. Roughly, these capture the representation data when restricted to the boundary, making them ripe for Dehn filling considerations. In this paper, we will focus on the holonomy extension locus. 

\medskip

Fix a meridian-longitude basis $(\mu, \ell)$ of $\pi_1(\partial M)$, where $\ell$ is the homological longitude. Let $\widetilde{\rho}: \pi_1(M) \to \widetilde{PSL_2(\mathbb{R})}$ have hyperbolic, parabolic, or central peripheral image, and let $\rho$ be its projection to $SL_2(\mathbb R)$. Choose a common invariant line $L \in \mathbb RP^1$ for $\rho(\pi_1(\partial M))$. For $\gamma \in \pi_1(\partial M)$, let $a_\gamma \in \mathbb R^\times$ be the eigenvalue of $\rho(\gamma)$ on $L$. Define
\begin{equation*}
	EV(\widetilde\rho, L) = (\log|a_\mu|, \log|a_\ell|, \mathrm{trans}(\widetilde\rho(\mu)), \mathrm{trans}(\widetilde\rho(\ell)))
\end{equation*} 

\begin{definition}
The \emph{holonomy extension locus} $HL(M)$ is the closure of these evaluation points in $\mathbb R^2 \times \mathbb Z^2$. Its \emph{zero-sheet} is 
\begin{equation*}
	H_{0,0}(M) = \{(u, v) \in \mathbb R^2 \mid (u, v, 0, 0) \in HL(M)\}
\end{equation*}
A point is \emph{realized} if it is the evaluation of an actual pair $(\widetilde\rho, L)$, rather than only a point added by taking closure. 
\end{definition}

It is shown in \cite{gao} that if $M$ is an irreducible rational homology solid torus, $H_{0, 0}(M)$ is a locally finite union of analytic arcs and isolated points in $\mathbb{R}^2$. Crucially, Lemma 3.8 in \cite{gao} shows that if $r = p/q$, $L_r = \{(u, v) \mid pu + qv = 0\}$, and $L_r \cap H_{0, 0}(M)$ contains a nonzero realized point and $M(r)$ is irreducible, then $\pi_1(M(r))$ is left-orderable. 

\subsection{Orderability from a real ideal point end}

Previous results, particularly those in \cite{gao}, identified smooth points on the character variety that have real coordinates, and then deformed these smooth points to arcs of $SL_2(\mathbb{R})$-representations. We will construct real representation arcs whose characters approach ideal points with real limiting character. The following criterion explains how such arcs yield intervals of left-orderable fillings.

\begin{theorem}\label{thm:main}
Let $M$ be a compact, connected, orientable, irreducible integral homology solid torus with meridian--longitude basis $(\mu,\ell)$. Let $D\subset X(M)$ be an irreducible curve and $x_\infty\in\widehat D$ an ideal point such that $I_\mu$ has a pole at $x_\infty$, $I_\ell$ is regular at $x_\infty$ with $I_\ell(x_\infty)\in\{2,-2\}$, and $I_\ell|_D$ is nonconstant. Suppose there is a continuous arc $\rho_t\colon\pi_1(M)\to SL_2(\mathbb{R})$, $0<t<\epsilon$, whose characters lie in $D$ and lift to an arc in $\widehat D$ converging to $x_\infty$, and that a lift of each $\bar\rho_t$ to $\widetilde{PSL}_2(\mathbb{R})$ has longitudinal translation number zero. Then $H_{0,0}(M)$ has a nonhorizontal end asymptotic to the horizontal axis, and $M$ is half-orderable near $0$.
\end{theorem}

\begin{proof}
Since $I_\mu$ has a pole at $x_\infty$, $|\mathrm{tr}(\rho_t(\mu))| \to \infty$. After shrinking the parameter interval, $\rho_t(\mu)$ is therefore hyperbolic. Its commuting peripheral image preserves each of its two real eigenlines. Choose the eigenline $L_t$ on which the meridional eigenvalue $a_\mu(t)$ satisfies $|a_\mu(t)| > 1$, and let $a_\ell(t)$ be the eigenvalue of $\rho_t(\ell)$ on the same line. These choices vary continuously. Set $u(t) = \log|a_\mu(t)|$ and $v(t) = \log|a_\ell(t)|$. Then $u(t) \to +\infty$. 
	
\medskip
	
Choose continuous lifts of the projectivized representations. Since $\mu$ generates $H_1(M; \mathbb Z)$ and $\ell$ is null-homologous, the lifts can be normalized to have meridional translation number zero without changing their longitudinal translation number (see Section 3 of \cite{cullerdunfield}). Thus $(u(t), v(t))$ is a realized point of $H_{0,0}(M)$ for every sufficiently small $t > 0$. 
	
\medskip
	
Since $I_\ell(\chi_{\rho_t})\to I_\ell(x_\infty)=\pm2$, both roots of $X^2-I_\ell(\chi_{\rho_t})X+1$ converge to $\pm1$; in particular $|a_\ell(t)|\to1$ and $v(t)\to0$. Since $I_\ell|_D$ is nonconstant, $I_\ell^2-4$ has isolated zeros on $\widehat D$, so after shrinking the interval $I_\ell(\chi_{\rho_t})^2\neq4$ for all $t$. Since $a_\ell(t)$ is real, $v(t)=0$ would force $a_\ell(t)=\pm1$ and hence $I_\ell(\chi_{\rho_t})^2=4$. Thus $v(t)\neq0$, so the end is nonhorizontal and $v(t)$ has constant sign. The continuous function $r(t) = -\frac{v(t)}{u(t)}$ is nonzero and tends to zero. Its image contains a one-sided interval with endpoint zero. 

\medskip
	
By Theorem 1.2 of \cite{reducible}, at most three filling slopes of $M$ yield reducible manifolds. Hence there is $a \neq 0$ such that every rational slope strictly between $0$ and $a$ lies in the image of $r(t)$ and gives an irreducible filling. For each $r = p/q$, the line $L_r = \{(u, v) \in \mathbb R^2 \mid pu + qv = 0\}$ meets the constructed end at a realized point with $u > 0$, hence at a nonzero realized point. It follows from Lemma 3.8 of \cite{gao} that $\pi_1(M(r))$ is left-orderable. 
\end{proof}

\section{Geometric representations}\label{sec:deformation}

We fix the topological setup for the rest of the paper. Let $M$ be an integral homology solid torus, and let $S$ be an essential nonfiber free genus one Seifert surface, i.e. one such that $M \setminus \setminus S$ is a genus-two handlebody $H$; these are the hypotheses of Theorem \ref{thm:mainthm}. Let $\pi_1(H) = \langle a, b \rangle$, the free group on two generators. Let $\partial H = S_1 \cup A \cup S_2$, where $S_1$ and $S_2$ are the two copies of $S$, and $A$ is the annulus inherited from cutting $\partial M$ along its homological longitude. Let $\iota_i: \pi_1(S_i) \hookrightarrow \pi_1(H)$ be the inclusions.

\medskip

Let $\hat S \subset M(0)$ be the torus obtained by capping $S$ with a meridian disk of the filling solid torus. Assume that $M(0)$ is irreducible, and that $\hat S$ is the unique JSJ torus of $M(0)$. Then $N = M(0) \setminus \setminus \hat S$ is the connected JSJ piece with two torus boundary components. Alternatively, one can construct $N$ by attaching a 2-handle to $H$ along $A$. 

\medskip

Use $[g, h] = ghg^{-1}h^{-1}$. Choose a based geometric basis $x, y$ of $\pi_1(S)$ so that $\delta = [x, y]$ represents the chosen oriented boundary loop. Choose a common basepoint in $A$ and basepoint paths in $A$ to the boundary basepoints of $S_1$ and $S_2$. Identifying both copies with the same oriented surface $S$ gives based embeddings $\iota_i: \pi_1(S) \to \pi_1(H)$, $i = 1, 2$ for which $\iota_1(\delta) = \iota_2(\delta) = \ell$ in $\pi_1(H)$. Here, the annular core represents the homological longitude $\ell$, with its orientation chosen consistently with $\delta$. Let
\begin{equation*}
	m_1 = \iota_1(x) \ \ \ \ \ \ell_1 = \iota_1(y) \ \ \ \ \ m_2 = \iota_2(x) \ \ \ \ \ \ell_2 = \iota_2(y)
\end{equation*}
Thus $[m_1, \ell_1] = [m_2, \ell_2] = \ell$. The isomorphism between the two surface subgroups is $\varphi = \iota_2 \circ \iota_1^{-1}$. The four words $m_1,\ell_1,m_2,\ell_2$ jointly generate $\pi_1(H)$: they are the images of generators of $\pi_1(\partial H)$, and $\pi_1(\partial H)\to\pi_1(H)$ is surjective.

\begin{figure}[h]
	\centering
	\includegraphics[scale=.6]{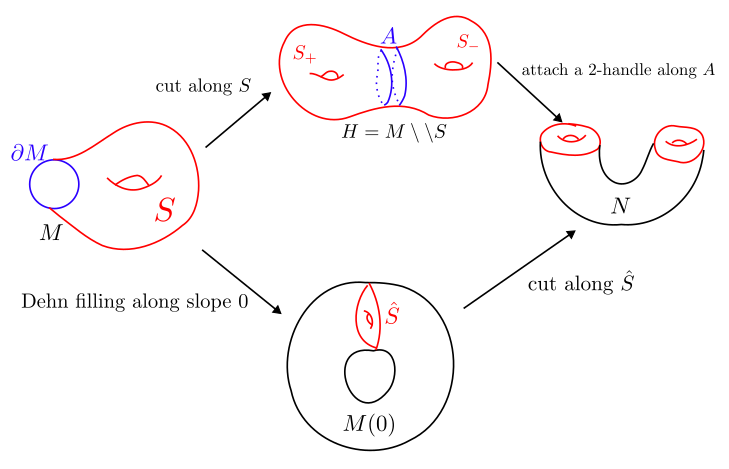}
\end{figure}

By the JSJ decomposition and Thurston's hyperbolization theorem for Haken manifolds, $N$ is either Seifert-fibered or hyperbolic. In addition, $N$ has two torus-boundary components. We have a quotient $\pi_1(N) = \pi_1(H) / \langle \langle \ell \rangle \rangle$.

\begin{lemma}\label{lma:notproduct}
$N$ is not homeomorphic to $T^2 \times I$.
\end{lemma}

\begin{proof}
Suppose $N \cong T^2 \times I$, so $\pi_1(H)/\langle \langle \ell \rangle \rangle \cong \mathbb Z^2$. Since $\ell = [m_1, \ell_1] \in [\pi_1(H), \pi_1(H)]$, the quotient map factors through a surjection $\pi_1(H)^{ab} \cong \mathbb Z^2 \to \mathbb Z^2$, which must be an isomorphism. Hence $\langle \langle \ell \rangle \rangle = [\pi_1(H), \pi_1(H)] = \langle \langle [a, b] \rangle \rangle$. By Magnus' theorem on normal closure in free groups \cite{mks}, $\ell$ is conjugate to $[a, b]^{\pm 1}$. Since $[m_i, \ell_i] = \ell$, Nielsen's criterion \cite{mks} shows that $\{m_i, \ell_i\}$ is a basis of $\pi_1(H)$ for $i = 1, 2$; that is, $\iota_1$ and $\iota_2$ are isomorphisms. Cutting $M$ along $S$ presents $\pi_1(M)$ as an HNN extension of $\pi_1(H)$ with associated subgroups $\iota_1(\pi_1(S))$ and $\iota_2(\pi_1(S))$. Since both equal $\pi_1(H)$, we get $\pi_1(M)\cong \pi_1(H)\rtimes\mathbb{Z}$, and the homomorphism $\pi_1(M) \to \mathbb Z$ dual to $S$ has kernel $\pi_1(H)$. As $M$ is irreducible, Stallings' fibration theorem \cite{stallings} shows that the class of $S$ is fibered. An incompressible surface homologous to a fiber is isotopic to it \cite{thurstonnorm}, so $S$ is a fiber, contrary to the hypothesis. 
\end{proof}

The quotient $\pi_1(N) = \pi_1(H) / \langle \langle \ell \rangle \rangle$ induces a closed embedding $X(N) \hookrightarrow X(\pi_1(H)) \cong \mathbb C^3$ whose image lies in the variety cut out by $V(I_\ell - 2)$. Since $\ell \neq 1$, $I_{\ell}-2$ is a nonzero polynomial, so $\dim X(N) \leq 2$. On the other hand $H_1(H; \mathbb Z) \cong \mathbb Z^2$, and its diagonal characters give a two-dimensional subvariety of $X(N)$. Hence $\dim X(N) = 2$. Thus, by Lemma 3.1 and Corollary 3.1 of \cite{wang}, $N$ is either hyperbolic or Seifert-fibred over $A^2(p)$, i.e. the annulus with a cone point of order $p \geq 2$.  

\medskip

To prove Theorem \ref{thm:mainthm}, we assume that $N$ is hyperbolic with a real place in its trace field or Seifert-fibered. Let $\pi: \pi_1(H) \twoheadrightarrow \pi_1(N)$ be the quotient map. Set $\bar\iota_i = \pi \circ \iota_i: \pi_1(S) \to \pi_1(N)$. 

\medskip

Define a representation $\rho: \pi_1(N) \to SL_2(\mathbb R)$ as follows. By Theorem 3.2 and its proof in \cite{wang}, the geometric holonomy of $N$ admits an $SL_2(\mathbb C)$-lift whose characters on the boundary tori agree under the gluing. (In the terminology of \cite{wang}, we say that the JSJ decomposition of $M(0)$ is \emph{$SL_2(\mathbb C)$-compatible}.) In the case where $N$ is hyperbolic with real place in its trace field, since $N$ is cusped, this lift can be conjugated into $SL_2(k)$, where $k$ is its trace field; see Section 8 of \cite{cullerdunfield}. For the chosen real embedding $\sigma: k \hookrightarrow \mathbb R$, put $\rho = \sigma \circ \widetilde\rho_{\mathrm{hol}}: \pi_1(N) \to SL_2(\mathbb R)$. The matching traces on the glued boundary tori are preserved by $\sigma$. In the Seifert case, let $\rho': \pi_1^{orb}(A^2(p)) \to PSL_2(\mathbb R)$ be the complete finite-area orbifold holonomy, with the boundary components corresponding to cusps, and let $q: \pi_1(N) \to \pi_1^{orb}(A^2(p))$ be the orbifold quotient. Choose a lift $\rho$ of $\rho' \circ q$ to $SL_2(\mathbb C)$ that is compatible, in the sense that $\chi_{\rho\circ\bar\iota_1} = \chi_{\rho\circ\bar\iota_2}$. Existence of such a lift, including equality of the peripheral sign characters, follows from Lemma 3.4 and the proof of Theorem 3.2 in \cite{wang}, applied to the one-vertex, one-edge JSJ graph of $M(0)$. Since the projective representation is real, this lift takes values in $SL_2(\mathbb R)$. In either case, put $\rho_0 = \rho \circ \pi$. Then $\chi_{\rho_0 \circ \iota_1} = \chi_{\rho_0 \circ \iota_2}$. 

\begin{prop}\label{prop:deformation}
The representation $\rho_0: \pi_1(H) \to SL_2(\mathbb R)$ is irreducible over $\mathbb C$. There exist matrices $Q_i \in SL_2(\mathbb R)$, a sign character $\epsilon: \pi_1(S) \to \{\pm 1\}$, and homomorphisms $h_i: \pi_1(S) \to \mathbb R$ such that
\begin{equation*}
	Q_i\rho_0(\iota_i(\gamma))Q_i^{-1} = \epsilon(\gamma)\begin{pmatrix}1&h_i(\gamma)\\0&1\end{pmatrix}.
\end{equation*}
The forms $h_1$, $h_2$ are linearly independent, and the invariant lines $L_i = Q_i^{-1}\langle e_1 \rangle$ are distinct. Moreover, there exist $\gamma_i \in \pi_1(S)$ with $h_i(\gamma_i) \neq 0$ such that
\begin{equation*}
	(dI_{\bar\iota_1(\gamma_1)}, dI_{\bar\iota_2(\gamma_2)}): H^1(\pi_1(N); \mathfrak g_{\rho}) \xrightarrow{\sim} \mathbb C^2
\end{equation*}
is an isomorphism.
\end{prop}

\begin{proof}
In the hyperbolic case, the first two assertions are direct consequences of the fact that $\rho: \pi_1(N) \to SL_2(\mathbb R)$ is a Galois conjugate of the holonomy representation. Let $\tau_i \in k$ be the shape parameter of each cusp of $N$, relative to the basis $(\bar\iota_i(x), \bar\iota_i(y))$. Since the two copies of the same oriented surface basis induce opposite boundary orientations on the two capped tori, we have $\mathrm{Im}\tau_1\mathrm{Im}\tau_2 < 0$. In particular, after applying the real embedding $\sigma$, we have $\sigma(\tau_2) - \sigma(\tau_1) \neq 0$. By definition, $\frac{h_i(y)}{h_i(x)} = \sigma(\tau_i)$, so $(h_1 \wedge h_2)(x, y) = h_1(x)h_2(x)\sigma(\tau_2 - \tau_1) \neq 0$. The isomorphism follows from Theorems 0.3, 0.4 of \cite{menalferrerporti}, as well as Corollary 1.3 of \cite{menalferrerporticoordinates}. These results give the isomorphism at the complex holonomy; the cocycle, coboundary, and trace differential maps are computed by finite linear systems over $k$, whose ranks are preserved by $\sigma$. The same isomorphism therefore holds at $\rho$, with $\gamma_1 = \gamma_2 = x$. 

\medskip

In the Seifert-fibered case, the first assertion follows from the fact that a finite-volume orbifold representation is irreducible. The second assertion is because this representation is nontrivial parabolic on the cusps of the base orbifold and equal to $\pm I$ on the regular fiber. Thus, the kernel of $h_i$ in $H_1(S; \mathbb R)$ is the line generated by the regular fiber. If $h_1, h_2$ were proportional, these two fiber lines would coincide under the gluing, which contradicts minimality of the JSJ decomposition. Now choose a simple closed curve $d_i$ meeting each regular fiber once; its image in the orbifold group generates its boundary subgroup. The adjoint module of the orbifold holonomy, i.e. $\mathfrak{sl}_2(\mathbb C)_{\mathrm{Ad}\rho'}$, pulls back to $\mathfrak g_{\rho}$. For any cocycle $z$ on $\pi_1(N)$, since the regular fiber is central and maps to $\pm I$, we get
\begin{equation*}
	z(f) = \mathrm{Ad}_{\rho(g)}z(f) \ \ \ \ \ g \in \pi_1(N)
\end{equation*}
Since $\rho_0$ is irreducible over $\mathbb C$, its adjoint module has no invariant vectors; thus $z(f) = 0$, and every cocycle descends to the orbifold quotient. We thus have $H^1(\pi_1^{orb}(A^2(p)); \mathfrak{sl}_2(\mathbb C)_{\mathrm{Ad}\rho'}) \cong H^1(\pi_1(N); \mathfrak g_{\rho})$.

\medskip

The orbifold holonomy is good and boundary-regular in the sense of \cite{portiorbifolds}, and its twisted Euler characteristic is $-2$. By Proposition 4.2 of \cite{portiorbifolds}, $\dim H^1(A^2(p); \mathfrak{sl}_2(\mathbb C)_{\mathrm{Ad}\rho'}) = 2$. The proof of Proposition 4.6 of \cite{portiorbifolds} shows that restriction to the boundary is surjective, and each boundary circle's cohomology is one-dimensional, so $H^1(A^2(p); \mathfrak{sl}_2(\mathbb C)_{\mathrm{Ad}\rho'}) \xrightarrow{\sim} \bigoplus_{i=1}^2H^1(\langle q(d_i) \rangle; \mathfrak{sl}_2(\mathbb C)_{\mathrm{Ad}\rho'})$. Conjugate so that $\rho(d_i) = \varepsilon_i\left(\begin{smallmatrix}1&t_i\\0&1\end{smallmatrix}\right)$, with $t_i \neq 0$. The trace differential is 
\begin{equation*}
	\mathfrak{sl}_2(\mathbb C)_{\mathrm{Ad}\rho'}/(\mathrm{Ad}_{\rho(d_i)}-1)\mathfrak{sl}_2(\mathbb C)_{\mathrm{Ad}\rho'} \to \mathbb C \ \ \ \ \ [X] \mapsto \mathrm{tr}(X\rho(d_i)) = \varepsilon_i t_i X_{21}
\end{equation*}
where $\varepsilon_i$ is the sign of the trace of $\rho(d_i)$. This is nonzero, hence an isomorphism since its domain is one-dimensional. Combining these maps with inflation and boundary restriction gives
\begin{equation*}
	(dI_{d_1},dI_{d_2}): H^1(\pi_1(N);\mathfrak g_{\rho}) \xrightarrow{\sim}\mathbb C^2.
\end{equation*}
Choose $\gamma_i \in \pi_1(S)$ with $\bar\iota_i(\gamma_i) = d_i$; this proves the desired isomorphism.

\medskip

In either case, a common fixed line for the two surface images would be preserved by the four words $m_1,\ell_1,m_2,\ell_2$, which jointly generate $\pi_1(H)$, contradicting irreducibility of $\rho_0$.
\end{proof}

\section{Real gluing and ideal point deformations}\label{sec:proof}

In this section, we finish the proof of Theorem \ref{thm:mainthm}, via the following result. The construction in this section uses only the representation-theoretic properties established in Proposition \ref{prop:deformation}. We state it in more general form before completing the proof of Theorem \ref{thm:mainthm}. 

\begin{theorem}\label{thm:endpointcriterion}
Let $M$ and $S$ satisfy the topological hypotheses of Theorem \ref{thm:mainthm}. Set $H = M \setminus \setminus S$, $N = M(0) \setminus \setminus \widehat S$. Suppose that $\rho: \pi_1(N) \to SL_2(\mathbb R)$ is irreducible over $\mathbb C$. Assume there are $Q_i \in SL_2(\mathbb R)$, a sign character $\epsilon: \pi_1(S) \to \{\pm 1\}$, and linearly independent homomorphisms $h_i: \pi_1(S) \to \mathbb R$ such that 
\begin{equation*}
	Q_i\rho(\bar\iota_i(\gamma))Q_i^{-1} = \epsilon(\gamma)\begin{pmatrix}1&h_i(\gamma)\\0&1\end{pmatrix} \ \ \ \ \ \gamma \in \pi_1(S)
\end{equation*}
with distinct invariant lines $L_i = Q_i^{-1}\langle e_1 \rangle$. Suppose, moreover, that there are $\gamma_i \in \pi_1(S)$ with $h_i(\gamma_i) \neq 0$ for which 
\begin{equation*}
	(dI_{\bar\iota_1(\gamma_1)}, dI_{\bar\iota_2(\gamma_2)}): H^1(\pi_1(N); \mathfrak g_\rho) \xrightarrow{\sim} \mathbb C^2
\end{equation*}
is an isomorphism. Then there is an irreducible algebraic curve $D \subset X(M)$ defined over $\mathbb R$ and a real analytic family 
\begin{equation*}
	\sigma_t: \pi_1(M) \to SL_2(\mathbb R) \ \ \ \ \ 0 < t < \epsilon_0
\end{equation*}
whose characters lie in $D$ and approach an ideal point detecting $S$, with 
\begin{equation*}
	\sigma_t|_{\pi_1(H)} \to \rho \circ \pi.
\end{equation*}
Every lift of the family has zero longitudinal translation number. Consequently, $H_{0,0}(M)$ contains a nonhorizontal end asymptotic to the horizontal axis, and $M$ is half-orderable near zero.
\end{theorem}

Fix an arbitrary $\rho$ satisfying the hypotheses of Theorem \ref{thm:endpointcriterion}, and let $\rho_0 = \rho \circ \pi$. 

\medskip

Let $F_2 = \langle a, b \rangle$ be the free group on two generators. For a representation $\rho: F_2 \to SL_2(\mathbb{C})$, let $X = \text{tr}(\rho(a)), Y = \text{tr}(\rho(b)), Z = \text{tr}(\rho(ab))$. Given a word $w \in F_2$, the trace function of $\rho(w)$ can be written as a polynomial $I_w(X, Y, Z)$ by applying the identity
\begin{equation*}
	\text{tr}(A)\text{tr}(B) = \text{tr}(AB) + \text{tr}(AB^{-1}) \ \ \ \ \ A, B \in SL_2(\mathbb{C})
\end{equation*}
These coordinates identify $X(F_2)$ with $\mathbb C^3$; this is a classical theorem of Fricke and Vogt (see \cite{goldman2}). In our situation, the complement of a genus one free Seifert surface is a genus two handlebody $H$, for which $\pi_1(H) = \langle a, b \rangle$, with $X(H) \cong \mathbb C^3$, with coordinates $X, Y, Z$, as described above.

\begin{definition}
Let $A_1, A_2 \leq G$, and let $\varphi: A_1 \to A_2$ be an isomorphism. The \emph{HNN extension of $G$ with respect to $\varphi$} is
\begin{equation*}
	G*_{\varphi} = \left\langle G, \tau \ \middle| \ \tau\gamma\tau^{-1} = \varphi(\gamma), \gamma \in A_1 \right\rangle
\end{equation*}
\end{definition} 

To extend a representation of $\pi_1(H)$ across the HNN splitting, its two surface restrictions must be conjugate by a matrix in $SL_2(\mathbb R)$. When these restrictions are irreducible over $\mathbb C$, equality of their characters gives a real conjugator.

\medskip

Choose the stable letter of the HNN extension by crossing the removed collar of $S$ along its boundary and closing the path through $A$. Choose its orientation and the closing path so that it represents the fixed meridian $\mu$. With the sides labeled accordingly, van Kampen's theorem gives
\begin{equation*}
	\pi_1(M) = \left\langle a, b, \tau \ \middle| \tau m_1\tau^{-1} = m_2, \tau\ell_1\tau^{-1} = \ell_2 \right\rangle \ \ \ \ \ \tau = \mu
\end{equation*}
In particular, the boundary-word relation is $\mu\ell\mu^{-1} = \ell$. This is known as the \emph{Lin presentation} \cite{lin}. 

\medskip

By the description of $N$, we have
\begin{equation*}
	N \cong H \cup_A(D^2 \times [0, 1]) \ \ \ \ \ \pi_1(N) \cong \pi_1(H)/\langle \langle \ell \rangle \rangle
\end{equation*}

\begin{lemma}\label{lma:matching}
Let $r: X(M) \to X(H)$ be the restriction map, and define 
\begin{equation*}
	\mathcal Z = V(I_{m_1} - I_{m_2}, I_{\ell_1} - I_{\ell_2}, I_{m_1\ell_1} - I_{m_2\ell_2}) \subset X(H) \ \ \ \ \ f = I_\ell - 2
\end{equation*}
Then every character in $\mathcal Z \setminus V(f)$ is the restriction of a character in $X(M)$. Moreover, each fiber of $r$ over $\mathcal Z \setminus V(f)$ contains at most two complex points.
\end{lemma}

\begin{proof}
Let $\chi \in \mathcal Z \setminus V(f)$. The three equations defining $\mathcal Z$ imply that the two surface restrictions have the same character, while $f(\chi) \neq 0$ says that their common commutator trace is different from 2, so both surface restrictions are irreducible. Irreducible representations with the same character are conjugate in $SL_2(\mathbb{C})$, so there is $T\in SL_2(\mathbb{C})$ with $T\rho(\iota_1(\gamma))T^{-1}=\rho(\iota_2(\gamma))$ for all $\gamma\in\pi_1(S)$. Sending $\tau\mapsto T$ extends $\rho$ across the
HNN extension; compare the proof of Theorem 1.1 of \cite{wang}. Thus $\chi$ lies in the image of $r: X(M) \to X(H)$. 

\medskip

It remains only to bound the fiber. Fix a representation $\rho: \pi_1(H) \to SL_2(\mathbb C)$ with character $\chi$. Since either surface restriction is irreducible, $\rho$ is irreducible, so every character in $r^{-1}(\chi)$ can (up to conjugation) be represented by an extension whose restriction to $\pi_1(H)$ is exactly $\rho$. Let $T, T' \in SL_2(\mathbb C)$ be two possible stable-letter matrices for extensions of $\rho$. Then $T\rho(m_1)T^{-1} = \rho(m_2)$ and $T\rho(\ell_1)T^{-1} = \rho(\ell_2)$, and similarly for $T'$. Then $T^{-1}T'$ centralizes the irreducible surface representation $\rho|_{\langle m_1, \ell_1 \rangle}$, and hence $T^{-1}T' = cI$ for some $c \in \mathbb C^\times$. Since $\det T = \det T' = 1$, we have $c^2 = 1$, so $T' = \pm T$. Thus the fiber of $r$ over $\chi$ contains at most two characters. 
\end{proof}

By assumption, there exist matrices $Q_i \in SL_2(\mathbb R)$, a sign character $\epsilon: \pi_1(S) \to \{\pm 1\}$, and homomorphisms $h_i: \pi_1(S) \to \mathbb R$ such that
\begin{equation*}
	Q_i\rho_0(\iota_i(\gamma))Q_i^{-1} = \epsilon(\gamma)\begin{pmatrix}1&h_i(\gamma)\\0&1\end{pmatrix}.
\end{equation*}
where $h_1$ and $h_2$ are linearly independent forms, and the invariant lines $L_i = Q_i^{-1}\langle e_1 \rangle$ are distinct. Define
\begin{equation*}
	\Omega = h_1(x)h_2(y) - h_1(y)h_2(x) \neq 0.
\end{equation*}

\begin{lemma}\label{lma:conjugators}
Define $\mathcal I = \{R \in M_2(\mathbb C) \mid R\rho_0(\iota_1(\gamma)) = \rho_0(\iota_2(\gamma))R$ for every $\gamma \in \pi_1(S)\}$; this is the set of matrices which conjugates one surface restriction to the other. Then $\mathcal I = \mathbb CP$, where 
\begin{equation*}
	P = Q_2^{-1}\begin{pmatrix}0&1\\0&0\end{pmatrix}Q_1
\end{equation*}
Moreover, $P$ is real and has rank one, with $\ker P = L_1$, $\im P = L_2$, and $\mathrm{tr}P \neq 0$. In particular, $\mathcal I \cap M_2(\mathbb R) = \mathbb RP$. 
\end{lemma}

\begin{proof}
For $R \in \mathcal I$, set $\widehat R = Q_2RQ_1^{-1}$. The defining equations for $\mathcal I$ become $h_1(\gamma)\widehat R\left(\begin{smallmatrix}0&1\\0&0\end{smallmatrix}\right) = h_2(\gamma)\left(\begin{smallmatrix}0&1\\0&0\end{smallmatrix}\right)\widehat R$ for $\gamma \in \langle x, y \rangle$. Since $\Omega \neq 0$, the equations for $x, y$ imply that $\widehat R\left(\begin{smallmatrix}0&1\\0&0\end{smallmatrix}\right) = \left(\begin{smallmatrix}0&1\\0&0\end{smallmatrix}\right)\widehat R = 0$, whose solution space is $\mathbb C\left(\begin{smallmatrix}0&1\\0&0\end{smallmatrix}\right)$. Therefore, $\mathcal I = \mathbb CP$, with $P$ in the desired form. The formula for $P$ gives $\ker P = L_1$, and $\im P = L_2$. A rank-one endomorphism of a two-dimensional vector space has trace zero precisely when its image equals its kernel. Since $L_1 \neq L_2$, the lemma follows. 
\end{proof}

For $\theta: G \to SL_2(\mathbb C)$, let $\mathfrak g_\theta = \mathfrak{sl}_2(\mathbb C)$ with action $\mathrm{Ad}\theta$. We use left logarithmic cocycles, i.e. for $z \in Z^1(G, \mathfrak{sl}_2(\mathbb C))$, $z(gh) = z(g) + \mathrm{Ad}_{\theta(g)}z(h)$. For $g \in G$, define the trace differential on cohomology by $dI_g([z]) = \mathrm{tr}(z(g)\theta(g))$. This is well-defined because it vanishes on coboundaries.

\begin{lemma}\label{lma:tangent}
For $v = [z] \in H^1(\pi_1(H); \mathfrak g_{\rho_0})$, let 
\begin{equation*}
	z_i(\gamma) = Q_iz(\iota_i(\gamma))Q_i^{-1} \ \ \ \ \ c_i(v)(\gamma) = (z_i(\gamma))_{21}.
\end{equation*}
Then $c_i(v): \pi_1(S) \to \mathbb C$ is a homomorphism depending only on $v$, and 
\begin{equation}\label{eq:tracederivative}
	dI_{\iota_i(\gamma)}(v) = \epsilon(\gamma)h_i(\gamma)c_i(v)(\gamma).
\end{equation}
There exist complex-linear functionals $A, B, \kappa$, defined over $\mathbb R$, such that
\begin{equation}\label{eq:tangentform}
	c_1(v) = A(v)h_1 + \kappa(v)h_2 \ \ \ \ \ c_2(v) = \kappa(v)h_1 + B(v)h_2
\end{equation}
\begin{equation}\label{eq:tangentcoords}
	H^1(\pi_1(N); \mathfrak g_\rho) = \ker \kappa \ \ \ \ \ (A, B, \kappa): H^1(\pi_1(H); \mathfrak g_{\rho_0}) \xrightarrow{\sim} \mathbb C^3
\end{equation}
is an isomorphism. Under the identification $H^1(\pi_1(H); \mathfrak g_{\rho_0}) \cong T_{\chi_0}X(H)$, we have
\begin{equation*}
	\bigcap_{\gamma \in \{x, y, xy\}}\ker d(I_{\iota_1(\gamma)} - I_{\iota_2(\gamma)})_{\chi_0} = \{A = B = 0\}.
\end{equation*}
It contains a unique real vector $v_0$ with $\kappa(v_0) = -1$, for which 
\begin{equation*}
	c_1(v_0) = -h_2 \ \ \ \ \ c_2(v_0) = -h_1.
\end{equation*}
\end{lemma}

\begin{proof}
Throughout the proof, denote $U(x) = \left(\begin{smallmatrix}1&x\\0&1\end{smallmatrix}\right)$.

\medskip

Since $\pi_1(H)$ is free of rank two, its cocycle space has dimension six. Absolute irreducibility makes the coboundary space three-dimensional, so $\dim_{\mathbb C}H^1(\pi_1(H); \mathfrak g_{\rho_0}) = 3$. By Proposition 4.4 of \cite{goldman}, the trace map has surjective differential at $\rho_0$; its kernel consists of infinitesimal conjugations. Thus, we have $H^1(\pi_1(H); \mathfrak g_{\rho_0}) = T_{\chi_0}X(H)$. Conjugation by an upper-unipotent matrix preserves the lower-left entry. The cocycle identity therefore makes each $c_i(v)$ a homomorphism on $\pi_1(S)$. Coboundaries have zero lower-left entry, so $c_i$ depends only on $v$. Taking the trace of $z_i(\gamma)\epsilon(\gamma)U(h_i(\gamma))$ gives \eqref{eq:tracederivative}. 

\medskip

Since $h_1$, $h_2$ form a basis of $\mathrm{Hom}(\pi_1(S), \mathbb C)$, write 
\begin{equation*}
	c_1 = Ah_1 + \kappa_1h_2 \ \ \ \ \ c_2 = \kappa_2h_1 + Bh_2
\end{equation*}
where the coefficients are linear functions of $v$. Since $\rho_0(\ell) = I$, the matrix $z(\ell)$ is independent of the representative cocycle. The commutator cocycle is thus
\begin{equation*}
	z_i([x, y]) = (1 - \mathrm{Ad}_{U(h_i(y))})z_i(x) + (\mathrm{Ad}_{U(h_i(x))} - 1)z_i(y).
\end{equation*}
Its lower-left entry is zero, and its upper-left entry is $h_i(x)c_i(y) - h_i(y)c_i(x)$. Since both boundary words equal $\ell$, the matrix $z(\ell)$ preserves $L_1$, $L_2$ and acts on them with eigenvalues $\kappa_1\Omega$ and $-\kappa_2\Omega$, respectively. These lines are distinct and $z(\ell)$ is traceless, so $\kappa_1 = \kappa_2$. Writing their common value as $\kappa$ gives \eqref{eq:tangentform}. 

\medskip

A cocycle on $\pi_1(H)$ descends to $\pi_1(N)$ precisely when it vanishes on $\ell$. Indeed, $\rho_0(\ell) = I$ and $z(\ell) = 0$ imply that $z$ vanishes on the normal closure of $\ell$. Coboundaries also vanish on $\ell$, so $H^1(\pi_1(N); \mathfrak g_\rho)$ is the kernel of the linear map $v \mapsto z(\ell)$. Since $z(\ell)$ preserves the distinct lines $L_1$, $L_2$ and has eigenvalues $\pm \kappa(v)\Omega$, it vanishes precisely when $\kappa(v) = 0$. Hence $H^1(\pi_1(N); \mathfrak g_\rho) = \ker \kappa$. On $H^1(\pi_1(N); \mathfrak g_\rho)$, \eqref{eq:tracederivative} becomes
\begin{equation*}
	dI_{\bar\iota_1(\gamma_1)} = \epsilon(\gamma_1)h_1(\gamma_1)^2A \ \ \ \ \ dI_{\bar\iota_2(\gamma_2)} = \epsilon(\gamma_2)h_2(\gamma_2)^2B.
\end{equation*}
The peripheral regularity hypothesis shows that $A, B$ are coordinates in the two-dimensional space $H^1(\pi_1(N); \mathfrak g_\rho)$. Since $\dim H^1(\pi_1(H); \mathfrak g_{\rho_0}) = 3$ and $H^1(\pi_1(N); \mathfrak g_\rho) = \ker \kappa$, Equation \eqref{eq:tangentcoords} follows. 

\medskip

Subtracting trace differentials gives $d(I_{\iota_1(\gamma)} - I_{\iota_2(\gamma)})(v) = \epsilon(\gamma)(A(v)h_1(\gamma)^2 - B(v)h_2(\gamma)^2)$. If these differences vanish for $\gamma = x, y, xy$, then $A(v)h_1^2 - B(v)h_2^2$ vanishes on the homology classes of $x$, $y$, $x + y$, so it vanishes on all of $H_1(S; \mathbb C)$. Since $h_1$ and $h_2$ are independent, this is equivalent to $A(v) = B(v) = 0$. Thus $\mathcal K = \{A = B = 0\}$. By \eqref{eq:tangentcoords}, its unique vector with $\kappa = -1$ is real and satisfies $c_1 = -h_2$, $c_2 = -h_1$, as desired.  
\end{proof}

\begin{lemma}\label{lma:curve}
The germ of $\mathcal Z$ at $\chi_0$ is a smooth complex curve germ, and is the germ of an irreducible algebraic curve $C \subset \mathcal Z$ defined over $\mathbb R$. Its real germ admits a regular real analytic parametrization $t \mapsto \chi_t$ with
\begin{equation*}
	\chi_{t=0} = \chi_0 \ \ \ \ \ \left.\frac{d}{dt}\chi_t\right|_{t=0} = v_0
\end{equation*} 
where $v_0$ is the vector in Lemma \ref{lma:tangent}. This parametrization can be realized by a real analytic family $\rho_t: \pi_1(H) \to SL_2(\mathbb R)$ through $\rho_0$, and
\begin{equation*}
	I_\ell(\chi_t) = 2 + \Omega^2t^2 + O(t^3).
\end{equation*}
In particular, $I_\ell|_C$ is nonconstant.
\end{lemma}

\begin{proof}
Let $\Sigma = h_1(x)h_2(y) + h_1(y)h_2(x)$. We first arrange so that $\Sigma \neq 0$. Changing the generating set $(x, y)$ to $(x, yx)$ preserves the boundary word, since $[x, yx] = [x, y]$, and preserves the full character-matching locus. Under this change, $h_i(y)$ is replaced by $h_i(y) + h_i(x)$, so $\Omega$ remains unchanged, and $\Sigma$ becomes $\Sigma + 2h_1(x)h_2(x)$. If $h_1(x)h_2(x) = 0$, then $\Sigma = \pm\Omega \neq 0$; otherwise, $\Sigma$ and $\Sigma + 2h_1(x)h_2(x)$ cannot both vanish. We therefore choose one of these two bases for which $\Sigma\neq0$, retaining the same notation for the generators and their images.

\medskip

Let $F_1 = I_{m_1} - I_{m_2}$, $F_2 = I_{\ell_1} - I_{\ell_2}$, and $V(F_1, F_2)$ the curve cut out by $F_1$ and $F_2$. By Lemma \ref{lma:tangent}, the Jacobian of $V(F_1, F_2)$ at $\chi_0$ is
\begin{equation*}
	d(F_1, F_2)_{\chi_0} = \begin{pmatrix}\epsilon(x)h_1(x)^2 & -\epsilon(x)h_2(x)^2&0 \\ \epsilon(y)h_1(y)^2 & -\epsilon(y)h_2(y)^2&0\end{pmatrix}
\end{equation*}
The minor in the first two columns has determinant $\epsilon(x)\epsilon(y)(-h_1(x)^2h_2(y)^2 + h_2(x)^2h_1(y)^2) = -\epsilon(x)\epsilon(y)\Omega\Sigma \neq 0$, so $V(F_1, F_2)$ is smooth at $\chi_0$, with tangent line $\{A = B = 0\}$. 

\medskip

Let $W = I_{m_1\ell_1} + I_{m_2\ell_2} - I_{m_1}I_{\ell_1}$. At $\chi_0$, we have $I_{m_i} = 2\epsilon(x)$, $I_{\ell_i} = 2\epsilon(y)$, and $I_{m_i\ell_i} = 2\epsilon(x)\epsilon(y)$. Thus, $W(\chi_0) = 0$. For a nonzero $\xi \in T_{\chi_0}V(F_1, F_2) = \{A = B = 0\}$, Lemma \ref{lma:tangent} gives $c_1(\xi) = \kappa(\xi)h_2$, $c_2(\xi) = \kappa(\xi)h_1$, and $\kappa(\xi) \neq 0$. Consequently, $dI_{m_i\ell_i}(\xi) = \epsilon(x)\epsilon(y)\kappa(\xi)(h_1(x) + h_1(y))(h_2(x) + h_2(y))$, for $i = 1, 2$. Hence,
\begin{equation*}
	dW(\xi) = 2\epsilon(x)\epsilon(y)\kappa(\xi)(h_1(x)h_2(y) + h_1(y)h_2(x)) = 2\epsilon(x)\epsilon(y)\kappa(\xi)\Sigma \neq 0.
\end{equation*}
Thus, $W$ is a local coordinate on $V(F_1, F_2)$ at $\chi_0$. 

\medskip

The common boundary word and the Fricke commutator identity give $W(I_{m_1\ell_1} - I_{m_2\ell_2}) = 0$ on $V(F_1, F_2)$; see the corresponding factorization in the proof of Theorem 1.1 in \cite{wang}. The local ring of the smooth curve $V(F_1, F_2)$ at $\chi_0$ is a domain, and $W$ is nonzero in that ring. Therefore $I_{m_1\ell_1} - I_{m_2\ell_2} = 0$ on the germ at $\chi_0$; it follows that the germs of $V(F_1, F_2)$ and $\mathcal Z$ at $\chi_0$ coincide. Let $C$ be the unique irreducible component of $\mathcal Z$ through $\chi_0$; then $C \subset \mathcal Z$. Since the equations are real and $\chi_0$ is real, complex conjugation preserves this unique component, so $C$ is defined over $\mathbb R$.

\medskip

We have $T_{\chi_0}C = \{A = B = 0\}$. Choose a regular real analytic parametrization $t \mapsto \chi_t$ of the real germ, rescaled so that its tangent is $v_0$. The real trace map has a local real analytic section through $\rho_0$, by Proposition 4.4 of \cite{goldman} and the real analytic submersion theorem. Use this section to obtain $\rho_t$ with $\chi_{\rho_t} = \chi_t$ with $\rho_{t=0} = \rho_0$. Write $\rho_t(\ell) = I + tB_\ell + t^2D_\ell + O(t^3)$. The commutator computation in Lemma \ref{lma:tangent} shows that $B_\ell$ acts on $L_1$, $L_2$ with eigenvalues $\kappa(v_0)\Omega = -\Omega$, $-\kappa(v_0)\Omega = \Omega$, respectively. Therefore, $\mathrm{tr}B_\ell = 0$, and $\det B_\ell = -\Omega^2$. Expanding $\det\rho_t(\ell) = 1$ through second order gives 
\begin{equation*}
	1 = 1 + t\mathrm{tr}B_\ell + t^2(\mathrm{tr}D_\ell + \det B_\ell) + O(t^3)
\end{equation*}
so $\mathrm{tr}D_\ell = \Omega^2$. Taking traces yields $I_\ell(\chi_t) = 2 + \Omega^2t^2 + O(t^3)$. Since $\Omega \neq 0$, the longitude trace is nonconstant.
\end{proof}

We are now ready to prove Theorem \ref{thm:endpointcriterion}.

\begin{proof}[Proof of Theorem \ref{thm:endpointcriterion}]
Take the curve $C \subset \mathcal Z$, the normalized parametrization $\chi_t$, and the real analytic family $\rho_t: \pi_1(H) \to SL_2(\mathbb R)$ from Lemma \ref{lma:curve}. Thus, $\rho_{t=0} = \rho_0$, $\chi_{\rho_t} = \chi_t$, and $I_\ell(\chi_t) = 2 + \Omega^2t^2 + O(t^3)$, and the tangent $v_0 = \dot\chi_0$ satisfies $c_1(v_0) = -h_2$, $c_2(v_0) = -h_1$. For sufficiently small $t$, $I_\ell(\chi_t) > 2$. 

\medskip

For $t > 0$, $\rho_t \circ \iota_1$ and $\rho_t \circ \iota_2$ have the same characters and are irreducible, so their real conjugating set is a line of invertible matrices through the origin in $M_2(\mathbb R)$. For $t = 0$, Lemma \ref{lma:conjugators} implies that the conjugating space is a line generated by a singular matrix $P$ with nonzero trace. The defining linear equations depend real analytically on $t$ and have constant rank, so there is a real analytic one-parameter family of conjugators $P_t$ for $t \geq 0$ through $P$. 

\medskip

Write $Q_2P_tQ_1^{-1} = \left(\begin{smallmatrix}0&1\\0&0\end{smallmatrix}\right) + tR + O(t^2)$. Differentiating the equations $P_t\rho_t(\iota_1(\gamma)) = \rho_t(\iota_2(\gamma))P_t$ and taking diagonal entries gives $h_2(\gamma)R_{21} = c_1(v_0)(\gamma) = -h_2(\gamma)$ and $h_1(\gamma)R_{21} = c_2(v_0)(\gamma) = -h_1(\gamma)$. Since $h_i \neq 0$, it follows that $R_{21} = -1$. Thus, for $t$ sufficiently small, $\det P_t = t + O(t^2) > 0$. Define a representation $\sigma_t: \pi_1(M) \to SL_2(\mathbb R)$ with $\sigma_t|_{\pi_1(H)} = \rho_t$ and $\sigma_t(\mu) = T_t = \frac{P_t}{\sqrt{\det P_t}}$. This satisfies the HNN relations and defines a real analytic family of irreducible representations of $\pi_1(M)$. 

\medskip

In this analytic arc, we have that as $t \to 0^+$, that $P_t \to P$, where the trace of $P$ is nonzero and $\det P_t \to 0^+$ as well. Thus, 
\begin{equation}\label{eq:meridian}
	|\mathrm{tr}(\sigma_t(\mu))| = \frac{|\mathrm{tr}P_t|}{\sqrt{\det P_t}} \to \infty,
\end{equation}
so the trace of the meridian blows up at $\chi_0$. Write $C^\circ = C \setminus V(I_\ell - 2)$. By Lemma \ref{lma:matching}, the restriction map $r^{-1}(C^\circ) \to C^\circ$ has finite fibers, so its irreducible components have dimension at most one. It contains the nonconstant character arc $\chi_{\sigma_t}$, whose restriction to $\pi_1(H)$ is $\chi_t$. The zero-dimensional components and the intersections of distinct one-dimensional components form a finite set. For sufficiently small $t$, $\chi_{\sigma_t}$ lies in an irreducible one-dimensional component of $r^{-1}(C^\circ)$. Let $D \subset X(M)$ be the Zariski closure of this component; then $r|_D$ is nonconstant and dominant onto $C$. Moreover, $D$ is defined over $\mathbb R$; for otherwise $D$ and its  complex conjugate would be distinct irreducible curves with infinitely many common real points. 

\medskip

Restriction induces a finite morphism between smooth projective models $\hat r: \hat D \to \hat C$. Let $c_0 \in \hat C$ be the unique point above the smooth point $\chi_0$. Shrinking the parameter interval to avoid singular points if necessary, lift $\chi_{\sigma_t}$ to an arc $\hat\sigma_t \in \hat D$. By finiteness of $\widehat r$, $\hat\sigma_t$ eventually lies in a neighborhood of a point $x_\infty \in \hat r^{-1}(c_0)$ and converges to it. Since $\hat\sigma_t$ is real, $x_\infty$ is real as well. If $x_\infty$ lay over an affine point of $D$, the regular function $I_\mu$ would have a finite limit along the lifted character arc, a contradiction, so $x_\infty$ is an ideal point of $D$. 

\medskip

For any $\gamma \in \pi_1(H)$, the trace function $\mathrm{tr}(\gamma)$ on $D$ is the pullback of the regular function on $C$, so it is regular at $x_\infty$. Thus, $I_\gamma(x_\infty) = \chi_0(\gamma)$, so $\chi_0$ is the limiting character on $\pi_1(H)$ at the ideal point $x_\infty$. By \eqref{eq:meridian}, $I_\mu$ has a pole at $x_\infty$. In addition, $I_\ell(x_\infty) = 2$; since $r|_D$ is dominant and $I_\ell|_C$ is nonconstant, $I_\ell|_D$ is also nonconstant. 

\medskip

Let $\mathcal T$ be the tree associated to $x_\infty$. Regularity of all traces on $\pi_1(H)$ implies that $\pi_1(H)$ fixes a vertex $v$, while the pole of $I_\mu$ makes $\mu$ act hyperbolically. The subgroup $A_2=\iota_2(\pi_1(S))=\mu\,\iota_1(\pi_1(S))\,\mu^{-1}$ fixes both $v$ and $\mu v$, hence the segment $[v,\mu v]$. Sending the vertex of the Bass-Serre tree of the HNN splitting stabilized by $\pi_1(H)$ to $v$, and its adjacent edge to this segment, defines an equivariant map after subdivision. By the surface-to-tree construction of Section 3.2 of \cite{tillmann}, parallel copies of $S$ form a surface associated to $x_\infty$; discarding redundant copies, $x_\infty$ detects $S$.

\medskip

It remains to prove zero longitudinal translation number. Lift $\sigma_t$ to an arc of representations $\tilde\sigma_t: \pi_1(M) \to \widetilde{PSL}_2(\mathbb R)$; this may be done because $H^2(\pi_1(M); \mathbb Z) = 0$ (see Section 3 of \cite{cullerdunfield}). Choose continuous lifts $U_t, V_t \in \widetilde{PSL}_2(\mathbb R)$ of $\rho_t(m_1)$ and $\rho_t(\ell_1)$ for $0 \leq t < \epsilon$, and set $L_t = [U_t, V_t]$. Changing either lift by a central element does not change this commutator. Hence, for every lift $\widetilde\sigma_t$ and every $t > 0$, $\widetilde\sigma_t(\ell) = L_t$. At $t = 0$, the image $\rho_0 \circ \iota_i(\pi_1(S))$ lies in a conjugate of the subgroup of upper triangular parabolic matrices; the lift of this subgroup in $\widetilde{PSL}_2(\mathbb R)$ is abelian, so $U_0$ and $V_0$ commute, and $L_0$ is the identity. On the other hand, $\mathrm{tr}\sigma_t(\ell) = I_\ell(\chi_t) > 2$ for sufficiently small $t > 0$. Thus $L_t$ is hyperbolic and its translation number is integral. By continuity of translation number, $\mathrm{trans}(\tilde\sigma_t(\ell)) = 0$ for every sufficiently small $t > 0$. 

\medskip

We have now constructed the curve $D$, the detecting ideal point $x_\infty$, and the real representation arc satisfying all the hypotheses of Theorem \ref{thm:main}. That theorem gives the nonhorizontal end of $H_{0,0}(M)$ asymptotic to the horizontal axis and the one-sided interval of left-orderable fillings. 
\end{proof}

\begin{proof}[Proof of Theorem \ref{thm:mainthm}]
Proposition \ref{prop:deformation} supplies a representation satisfying the hypotheses of Theorem \ref{thm:endpointcriterion}. The conclusion follows.
\end{proof}

We also determine the side of the interval Theorem \ref{thm:mainthm} shows to be orderable.

\begin{corollary}\label{cor:side}
Fix the real representation $\rho_0$ and retain the conventions
\begin{equation*}
	\mu\iota_1(\gamma)\mu^{-1} = \iota_2(\gamma) \ \ \ \ \ \ell = [m_1, \ell_1].
\end{equation*}
For the end of $H_{0,0}(M)$ constructed in Theorem \ref{thm:mainthm}, the interval of left-orderable fillings can be chosen on the side opposite to $\Omega = h_1(x)h_2(y) - h_1(y)h_2(x)$. Equivalently, the constant $a$ in that theorem can be chosen with $a\Omega < 0$. 
\end{corollary}

\begin{proof}
Take $\sigma_t$ as in the proof of Theorem \ref{thm:endpointcriterion}, so that $\sigma_t(\mu) = \frac{P_t}{\sqrt{\det P_t}}$, and $\det P_t = t + O(t^2)$. The rank-one matrix $P = P_{t=0}$ has distinct eigenvalues $0$ and $p_0 = \mathrm{tr}P \neq 0$, and its $p_0$-eigenline is $\im P = L_2$. Let $p(t)$ and $\Lambda_t$ be the real analytic continuations of this eigenvalue and eigenline. The meridional eigenvalue on $\Lambda_t$ is $a_\mu(t) = \frac{p(t)}{\sqrt{\det P_t}}$, whose absolute value tends to infinity. Thus, $\Lambda_t$ is the expanding meridional eigenline. 

\medskip

Since $P_t$ commutes with $\rho_t(\ell)$ (apply $P_t\rho_t(\iota_1(\gamma)) = \rho_t(\iota_2(\gamma))P_t$ for $\gamma = \delta$), the longitude preserves $\Lambda_t$. By the calculation in Lemma \ref{lma:curve}, $\rho_t(\ell) = I + tB_\ell + O(t^2)$, with $B_\ell|_{L_2} = \Omega$. Because $\Lambda_0 = L_2$ and $\rho_0(\ell) = I$, its longitudinal eigenvalue is therefore $a_\ell(t) = 1 + \Omega t + O(t^2)$. For sufficiently small $t > 0$, the filling slopes $r(t) = -v(t)/u(t)$ have sign $-\mathrm{sgn}\Omega$. The argument of Theorem \ref{thm:main} gives the claimed interval. 
\end{proof}

\begin{remark}
The sign of $\Omega$ depends only on $\rho$ and on the orientations of $\mu$ and $\ell$. Indeed, each $Q_i$ is determined up to left multiplication by an upper triangular matrix in $SL_2(\mathbb{R})$, which rescales $h_i$ by a positive constant, and a change of basis $(x,y)$ preserving $[x,y]$ up to conjugacy acts on $H_1(S)$ with determinant one. In the hyperbolic case $\Omega$ depends on the chosen real place.
\end{remark}

\section{Examples and further questions}\label{sec:examples}

\subsection{Odd pretzel knots}

As explicit examples to which Theorem \ref{thm:mainthm} applies, we turn to odd classical pretzel knots. These pretzel knots have free Seifert genus one and are nonfibered, making them natural candidates to apply the theorem. Following Definition 2.1 of \cite{sekino}, let $P_{l, m, n}$ denote the pretzel link whose standard diagram has $l$, $m$, and $n$ right half-twists, so that $P_{-l, -m, -n}$ is its mirror image. Slopes are taken with respect to the meridian and canonical longitude. Passing to the mirror image replaces $r$ by $-r$ in every statement about filling slopes. 

\medskip

Let $M_{2p+1, 2q+1, 2r+1}$ be the complement of $P_{2p+1, 2q+1, 2r+1}$. The JSJ decompositions of $M_{2p+1, 2q+1, 2r+1}(0)$ are well-understood, by following theorem of \cite{sekino}. In the alternatives below, the exceptional knot families are understood up to permutation of the pretzel parameters and mirror image. We consider odd pretzel knots with $p, q, r \notin \{-1, 0\}$, so none of the three twist parameter is $\pm 1$, so cases (1.1.1)-(1.1.4) of the cited theorem in \cite{sekino} are excluded. 

\begin{theorem}[\cite{sekino}]\label{thm:sekino}
We have the following descriptions of JSJ decompositions of $M_{2p+1, 2q+1, 2r+1}(0)$. 
\begin{itemize}
	\item $M_{-3, 3, 2n+1}(0)$ has one JSJ component which is the exterior of the $(2, 4)$ torus link, which is Seifert fibered over the annulus with a cone point of order 2.
	\item $M_{-3, 5, 5}(0)$ and $M_{3, -5, -5}(0)$ has two JSJ components: the trefoil knot complement and the trivial circle fiber over the thrice-punctured sphere, denoted $S^1 \times S_{0, 3}$. 
	\item For any other $(p, q, r)$, $M_{2p+1, 2q+1, 2r+1}(0)$ has one JSJ component which is a hyperbolic 3-manifold. 
\end{itemize}
\end{theorem}

\begin{remark}
The JSJ decomposition of $M_{-3, 5, 5}(0)$ and $M_{3, -5, -5}(0)$ suggests that there is a separating twice-punctured torus with slope 0. This becomes the torus separating the trefoil knot complement from $S^1 \times S_{0, 3}$, while the Seifert surface caps off to the essential torus bounding the two other torus boundary components of $S^1 \times S_{0, 3}$. In \cite{twicepunctured}, it is shown that the union of the Seifert-surface and the twice-punctured torus is detected by an ideal point, where the limiting character at that ideal point restricts to the holonomy representation of the thrice-punctured sphere and the $(2, 3, \infty)$ triangle orbifold underlying the trefoil knot complement. 
\end{remark}

Theorem 1.11 of Le's thesis \cite{lethesis} established the existence of a one-sided interval of left-orderable fillings near zero for $M_{-3, 3, 2n+1}$, without specifying the side in the theorem's statement. Here we construct an ideal-point end detecting the standard Seifert surface and show that it supplies positive filling slopes in the peripheral convention fixed above. For the family $M_{-3, 3, 2n+1}$, we determine the side of $0$ the resulting half-orderable interval lies on.

\begin{prop}\label{prop:side-seifert}
For every $n \geq 1$ there exists $a_n > 0$ such that $\pi_1(M_{-3, 3, 2n+1}(r))$ is left-orderable for every rational slope $0 < r < a_n$. 
\end{prop}

\begin{proof}
Here $(p, q, r) = (-2, 1, n)$. By Section 2.2 of \cite{sekino}, 
\begin{equation*}
	\pi_1(M) = \langle a, b, t \mid ta^{n+1}babt^{-1} = a^{n+1}ba, tb^{-1}abt^{-1} = b^{-1}aba \rangle
\end{equation*}
where $a, b$ generate $\pi_1(H)$, $t$ is the meridian, and the canonical longitude is $\lambda_n = b^{-2}(ba)^2a^nb^2(b^{-1}a^{-1})^2a^{-n}$. Choose the basis $(x, y)$ of $\pi_1(S)$ for which
\begin{equation*}
	m_1 = b^{-1}ab \ \ \ \ \ \ell_1 = a^{n+1}bab \ \ \ \ \ m_2 = b^{-1}aba \ \ \ \ \ \ell_2 = a^{n+1}ba
\end{equation*}
Then $\mu\iota_1(\gamma)\mu^{-1} = \iota_2(\gamma)$ with $\mu = t$, and free reduction gives $[m_1, \ell_1] = [m_2, \ell_2] = \lambda_n$, so the conventions of Sections \ref{sec:deformation} and \ref{sec:proof} hold with $\ell = \lambda_n$. Let 
\begin{equation*}
	\rho_n(a) = \begin{pmatrix}-1&1\\0&-1\end{pmatrix} \ \ \ \ \ \rho_n(b) = \begin{pmatrix}2n+1&n\\2&1\end{pmatrix}.
\end{equation*}
Then $\rho_n(\lambda_n) = I$, so $\rho_n$ factors through $\pi_1(N)$. Its image in $PSL_2(\mathbb R)$ is generated by the parabolic $\rho_n(a)$ and the element $\rho_n(a^{n+1}b)$ of order two; $\rho_n$ is the lift of the holonomy of $A^2(2)$ composed with the orbifold quotient. We have 
\begin{equation*}
	\begin{pmatrix}1&0\\2&1\end{pmatrix}\rho_n(m_1)\begin{pmatrix}1&0\\2&1\end{pmatrix}^{-1} = \begin{pmatrix}-1&1\\0&-1\end{pmatrix} \ \ \ \ \ \begin{pmatrix}1&0\\2&1\end{pmatrix}\rho_n(\ell_1)\begin{pmatrix}1&0\\2&1\end{pmatrix}^{-1} = (-1)^{n+1}\begin{pmatrix}1&n\\0&1\end{pmatrix}
\end{equation*}
\begin{equation*}
	\begin{pmatrix}1&-1\\1&0\end{pmatrix}\rho_n(m_2)\begin{pmatrix}1&-1\\1&0\end{pmatrix}^{-1} = \begin{pmatrix}-1&-4\\0&-1\end{pmatrix} \ \ \ \ \ \begin{pmatrix}1&-1\\1&0\end{pmatrix}\rho_n(\ell_2)\begin{pmatrix}1&-1\\1&0\end{pmatrix}^{-1} = (-1)^{n+1}\begin{pmatrix}1&2\\0&1\end{pmatrix}
\end{equation*}
Thus, the signs agree on the two sides, so $\rho_n$ is the compatible lift of Section \ref{sec:deformation}, and $\Omega = h_1(x)h_2(y) - h_1(y)h_2(x) = (-1)(2) - n \cdot 4 = -(4n+2) < 0$. By Corollary \ref{cor:side}, the interval lies at positive slopes. 
\end{proof}

For $K = P_{2p+1, 2q+1, 2r+1}$, the coefficient $d$ in the normalized Alexander polynomial is $d = (p+1)(q+1)(r+1) - pqr$; see Lemma 2.1 of \cite{sekino}. In particular, for $K_n = P_{-3, 3, 2n+1}$, we have $d = -2$, and $\Delta_{K_n}(t) = 5 - 2(t + t^{-1})$, and $t\Delta_{K_n}(t) = -(2t-1)(t-2)$. Its roots are $1/2$ and $2$, so the Herald-Zhang criterion \cite{heraldzhang} does not apply. The preceding proposition obtains a one-sided interval of left-orderable fillings from the Seifert geometry of the capped complement instead. It also identifies a real ideal-point end of $H_{0,0}(M)$ detecting the standard Seifert surface. In addition, unlike the results of Gao \cite{gao} utilizing real roots of the Alexander polynomial, our results do not require longitudinal rigidity.

\medskip

We now consider the other case of Theorem \ref{thm:mainthm}, i.e. when the geometry of the capped off complementary region is hyperbolic.  

\begin{corollary}\label{cor:main}
Let $K$ be an odd pretzel knot in the hyperbolic case of Theorem \ref{thm:sekino}. If the trace field of $N_K$ has a real place, then $M_K$ is half-orderable near $0$.
\end{corollary}

Utilizing SnapPy and Regina, we record some explicit odd pretzel knots to which Corollary \ref{cor:main} does and does not apply. For the odd pretzel knots considered above, let $M_K$ denote the knot exterior and let $N_K=M_K(0)\setminus\!\setminus\widehat S$ be the complement of the capped standard Seifert surface. Table \ref{tab:hyperbolic-examples} records the computed identifications of these hyperbolic capped complements, together with defining polynomials for their trace fields. Here $r_1$ denotes the number of real embeddings, equivalently the number of real roots of the defining polynomial. The census labels refer to $N_K$, not to $M_K$.

\medskip

The computations were performed using SnapPy and Regina within SageMath. Each pretzel knot was constructed as the numerator closure of a sum of three rational tangles. With SnapPy's conventions for rational tangles, this construction yields the mirror image of $P_{l,m,n}$ as defined above. The identifications and trace fields in Table \ref{tab:hyperbolic-examples} are insensitive to orientation, but filling slopes computed from these triangulations have the opposite sign to those in Proposition \ref{prop:side-seifert}. We formed its exterior in SnapPy and performed $(0,1)$-Dehn filling in the standard meridian-longitude coordinates. The filled triangulation was then converted to Regina. The identification of the resulting hyperbolic piece with $N_K$ uses the single-torus decomposition established in \cite{sekino}. Trace-field polynomials were obtained from the identified census manifolds using \texttt{trace\_field\_gens().find\_field} with precision $200$ bits, degree bound $20$, and \texttt{optimize=True}. The real embeddings of each resulting number field were computed using SageMath's \texttt{real\_places()}. The trace-field identifications rely on numerical recognition; the real-place counts are those of the number fields defined by the returned polynomials.

\begin{table}[htbp]
	\centering
	\small
	\setlength{\tabcolsep}{5pt}
	\renewcommand{\arraystretch}{1.3}
	\begin{tabular}{@{}lclc@{}}
		\toprule
		$K$ & $N_K$ & Trace-field polynomial & $r_1$ \\
		\midrule
		$P_{3, 3, 3}$
		& \texttt{m202}
		& $x^2+3$
		& $0$ \\
		
		$P_{3, 3, 5}$
		& \texttt{m357}
		& $x^4-x^2-3x-2$
		& $2$ \\
		
		$P_{3, 5, 5}$
		& \texttt{s876}
		& $x^4-3x^2-x+2$
		& $2$ \\
		
		$P_{3, 3, 7}$
		& \texttt{s548}
		& $x^5-4x^3-x^2+5x+1$
		& $1$ \\
		
		$P_{3, 5, 7}$
		& \texttt{v2601}
		& $\begin{aligned}
			&x^8-x^7-7x^6+7x^5+17x^4\\[-2pt]
			&\qquad{}-17x^3-16x^2+13x+2
		\end{aligned}$
		& $4$ \\
		
		$P_{5, 5, 5}$
		& \texttt{v3461}
		& $x^4 + 2x^2 - 3x + 1$
		& $0$ \\
		
		$P_{-3, 5, 7}$
		& \texttt{m129}
		& $x^2+1$
		& $0$ \\
		
		$P_{-3, 5, 9}$
		& \texttt{m295}
		& $x^3-x^2+x-2$
		& $1$ \\
		
		$P_{-3, 5, 11}$
		& \texttt{s443}
		& $x^4-x^3+2x^2-2x+2$
		& $0$ \\
		
		$P_{-3, 5, 13}$
		& \texttt{v1060}
		& $x^5-x^4-4x^3+3x^2+2x+2$
		& $1$ \\
		
		  $P_{-5, 5, 5}$
		& \texttt{s906}
		& $x^4-7x^2-5x-1$
		& $2$ \\
		\bottomrule
	\end{tabular}
	\caption{Hyperbolic capped complements of selected odd
		pretzel knots and their trace-field data.}
	\label{tab:hyperbolic-examples}
\end{table}

The numerically recognized fields have real places for the following knots:
\begin{equation*}
	P_{3,3,5},\quad P_{3,5,5},\quad P_{3,3,7},\quad
	P_{3,5,7},\quad P_{-3,5,9},\quad P_{-3,5,13},\quad
	P_{-5,5,5}.
\end{equation*}
Assuming the computed trace-field identifications, Theorem \ref{thm:mainthm} applies to each of these knot exteriors. For the rows with $r_1 = 0$, the numerically recognized trace fields have no real embeddings. Thus these computations do not provide the real-holonomy input required by Theorem \ref{thm:mainthm}. This does not rule out real ideal-point deformations arising from other representations of $\pi_1(N_K)$. In particular, the numerical holonomy extension locus of the $P_{3,3,3}$ exterior appears to contain a nonhorizontal arc in $H_{0,0}(M_K)$ asymptotic to the horizontal axis, although its capped complement $m202$ has trace field $\mathbb Q(\sqrt{-3})$. This suggests investigating real representations of $\pi_1(\texttt{m202})$ other than Galois conjugates of complete holonomy as possible limiting representations for the observed end.

\medskip

Among these seven proposed applications, all except $P_{-5, 5, 5}$ have $d > 0$, so Herald-Zhang \cite{heraldzhang} already gives two-sided orderability near zero for those six knot exteriors. Their role here is to illustrate the geometric construction of detecting ideal-point ends, rather than to provide new local orderability conclusions. For $P_{-5, 5, 5}$, we instead have $d = -6$ and $\Delta_K(t) = 13 - 6(t + t^{-1})$, and $t\Delta_K(t) = -(3t-2)(2t-3)$. Thus, its Alexander roots are $2/3$ and $3/2$, and this example lies outside the Herald-Zhang criterion. 

\subsection{Further problems}

We suggest natural further directions of study motivated by the present paper. 

\medskip

The following examples suggest that the methods in this paper could generalize to higher-genus Seifert surfaces. 

\begin{example}[$7_3$]
The knot $7_3$ is non-fibered and has genus 2. Figure 4 of \cite{gao} displays a numerical holonomy extension locus with an arc, distinct from the horizontal axis, having an apparent horizontal end. The caption notes that the theorems in that paper do not explain this arc. To account for it using the present approach, one would need to identify a detecting ideal point with the required limiting real representation, construct a real branch approaching that same ideal point, and verify zero longitudinal translation number along the branch. These additional verifications are not supplied here.
\end{example}

\begin{example}[$8_6$]
The knot $8_6$ is also nonfibered and has genus two. The numerical plot in Figure 6 of \cite{gaotwobridge} similarly suggests an additional arc with a horizontal end. It raises the same questions about the limiting representation, real gluing, and longitudinal translation number. Since its Seifert surface has genus two, the free genus-one construction of this paper does not apply directly. 
\end{example}

In order to further generalize these results for genus greater than one, we need to understand limiting characters of these detected Seifert surfaces. To begin, we pose the following question.

\begin{question}
Which real limiting characters occur at ideal points detecting Seifert surfaces of genus greater than one? What conditions on these characters ensure the existence of an approaching real representation arc with zero longitudinal translation number?
\end{question}

The author observed the following interesting patterns in his computational experiments on character varieties of odd pretzel knots.
\begin{itemize}
	\item For the $(-3, 3, 2n+1)$ knot complement, the genus of the Zariski-closure of the restriction image of canonical component appears to be equal to $n$. This, combined with the Riemann-Hurwitz theorem, provides a lower bound for the genus of the canonical component of such pretzel knots. This could lead to effective bounds on the genera of canonical components, which are largely unknown outside of two-bridge knots \cite{petersen} and once-punctured torus bundles of tunnel number one \cite{bakerpetersen}. 
	\item For $n = 3, 4, 5, 6$, the degree of the trace field of JSJ component on 0-surgery on the $(-3, 5, 2n+1)$ knot has degree $n - 1$. If the same formula holds for all $n \geq 3$, the JSJ components of the 0-surgeries of the $(-3, 5, 4n+1)$ knot complements would have odd degree trace field, meaning they have real places, and they would be another infinite family of odd pretzel knots to which Theorem \ref{thm:mainthm} applies. For $n \geq 2$, these knots have $d = 2n-3 > 0$, so their two-sided local orderability already follows from \cite{heraldzhang}. Proving the proposed trace field formula would instead establish an infinite hyperbolic family to which the detecting ideal-point construction of this paper applies.
\end{itemize}

A further problem is to determine how far the constructed real branch continues and which filling slopes its peripheral evaluations realize. This requires understanding its possible affine and ideal endpoints, including irreducible boundary-parabolic representations of zero longitudinal height. In particular:

\begin{question}
When can the local interval supplied by the theorem be extended to an interval bounded by other detected boundary slopes?  
\end{question}

The techniques in this paper become more subtle when considering incompressible surfaces whose boundary slopes are not 0. For example, for the figure-eight knot, the first example in Section 4 of \cite{gao} shows that there is an arc on the holonomy extension locus resulting in an interval of orderable slopes between -4 and 4, indicating that there is an arc of $SL_2(\mathbb{R})$-representations coming from both ideal points on its canonical component. (These ideal points detect twice-punctured tori with slope $\pm 4$.)

\medskip

However, consider the $(-2, 3, 7)$ pretzel knot. In \cite{gordonluecke}, it was shown that there is an incompressible twice-punctured torus with slope 37/2, and that 37/2-Dehn filling admits one JSJ torus. The JSJ complementary regions are fibered over the disk with two cone points. This suggests looking for an ideal point detecting this surface and an approaching arc of $SL_2(\mathbb R)$-representations. In \cite{niepretzel}, it was shown that 37/2 surgery on the $(-2, 3, 7)$ pretzel knot is not left-orderable. The known nonorderability results constrain any real deformation arising from this decomposition. Understanding whether such a deformation exists, and how its peripheral translation numbers behave, would clarify the additional obstructions at nonlongitudinal boundary slopes. 

\begin{question}
Let $M$ be an integral homology solid torus with homological longitude $\ell$. For which boundary slopes detected by ideal points can we establish lifts of $SL_2(\mathbb{R})$ representations lifting to $\widetilde{PSL_2(\mathbb{R})}$-representations $\widetilde{\rho_t}$ such that $\text{trans}(\widetilde{\rho_t}(\ell)) = 0$?
\end{question}

The author's preprint \cite{conway} begins to approach this question for the slope $\infty$.

\bibliography{sl2r.bib}
\bibliographystyle{plain}
	
\end{document}